\documentclass{amsart}

\usepackage{amssymb}

\usepackage{tikz}
\usepackage{graphicx}

\theoremstyle{plain}
\newtheorem{theorem}{Theorem}
\newtheorem{lemma}[theorem]{Lemma}
\newtheorem{corollary}[theorem]{Corollary}
\newtheorem{proposition}[theorem]{Proposition}

\theoremstyle{definition}
\newtheorem{definition}[theorem]{Definition}
\newtheorem{example}[theorem]{Example}
\newtheorem{question}[theorem]{Question}

\theoremstyle{remark}
\newtheorem{remark}[theorem]{Remark}
\newtheorem*{acknowledgments}{Acknowledgments}

\tikzset{myarrow/.style={->,>=stealth,line width=0.5pt}}

\DeclareMathOperator{\disc}{disc}
\DeclareMathOperator{\Lat}{Lat}
\DeclareMathOperator{\Mult}{Mult}
\DeclareMathOperator{\res}{res}

\begin{document}

\title{Quadratic rational maps with integer multipliers}

\author{Valentin Huguin}
\address{Institut de Math\'{e}matiques de Toulouse, UMR 5219, Universit\'{e} de Toulouse, CNRS, UPS, F-31062 Toulouse Cedex 9, France}
\email{valentin.huguin@math.univ-toulouse.fr}


\begin{abstract}
In this article, we prove that every quadratic rational map whose multipliers all lie in the ring of integers of a given imaginary quadratic field is a power map, a Chebyshev map or a Latt\`{e}s map. In particular, this provides some evidence in support of a conjecture by Milnor concerning rational maps whose multipliers are all integers.
\end{abstract}

\maketitle

\section{Introduction}

Given a rational map $f \colon \widehat{\mathbb{C}} \rightarrow \widehat{\mathbb{C}}$ and a point $z_{0} \in \widehat{\mathbb{C}}$, we study the sequence $\left( f^{\circ n}\left( z_{0} \right) \right)_{n \geq 0}$ of iterates of $f$ at $z_{0}$. The set $\left\lbrace f^{\circ n}\left( z_{0} \right) : n \geq 0 \right\rbrace$ is called the \emph{forward orbit} of $z_{0}$ under $f$.

The point $z_{0}$ is said to be \emph{periodic} for $f$ if there exists an integer $n \geq 1$ such that $f^{\circ n}\left( z_{0} \right) = z_{0}$; the least such integer $n$ is called the \emph{period} of $z_{0}$. The forward orbit of $z_{0}$, which has cardinality $n$, is said to be a \emph{cycle} for $f$. The \emph{multiplier} of $f$ at $z_{0}$ is the unique eigenvalue of the differential of $f^{\circ n}$ at $z_{0}$. The map $f$ has the same multiplier at each point of the cycle.

The multiplier is invariant under conjugacy: if $f$ and $g$ are rational maps, $\phi$ is a M\"{o}bius transformation such that $\phi \circ f = g \circ \phi$ and $z_{0}$ is a periodic point for $f$, then $\phi\left( z_{0} \right)$ is a periodic point for $g$ with the same period and the same multiplier.

We wish to examine here the rational maps that have only integer multipliers.

\begin{definition}
A rational map $f \colon \widehat{\mathbb{C}} \rightarrow \widehat{\mathbb{C}}$ of degree $d \geq 2$ is said to be a \emph{power map} if it is conjugate to $z \mapsto z^{\pm d}$.
\end{definition}

For every $d \geq 2$, there exists a unique polynomial $T_{d} \in \mathbb{C}[z]$ such that \[ T_{d}\left( z +z^{-1} \right) = z^{d} +z^{-d} \, \text{.} \] The polynomial $T_{d}$ is monic of degree $d$ and is called the $d$th \emph{Chebyshev polynomial}.

\begin{definition}
A rational map $f \colon \widehat{\mathbb{C}} \rightarrow \widehat{\mathbb{C}}$ of degree $d \geq 2$ is said to be a \emph{Chebyshev map} if it is conjugate to $\pm T_{d}$.
\end{definition}

\begin{remark}
For every $d \geq 2$, the rational maps $-T_{d}$ and $T_{d}$ are conjugate if and only if $d$ is even.
\end{remark}

These rational maps share the following well-known property:

\begin{proposition}
\label{proposition:specialPC}
Suppose that $f \colon \widehat{\mathbb{C}} \rightarrow \widehat{\mathbb{C}}$ is a power map or a Chebyshev map. Then $f$ has only integer multipliers.
\end{proposition}

In fact, there exist also other rational maps that satisfy this special condition.

\begin{definition}
A rational map $f \colon \widehat{\mathbb{C}} \rightarrow \widehat{\mathbb{C}}$ of degree $d \geq 2$ is said to be a \emph{Latt\`{e}s map} if there exist a torus $\mathbb{T} = \mathbb{C}/\Lambda$, with $\Lambda$ a lattice in $\mathbb{C}$, a holomorphic map $L \colon \mathbb{T} \rightarrow \mathbb{T}$ and a nonconstant holomorphic map $p \colon \mathbb{T} \rightarrow \widehat{\mathbb{C}}$ that make the following diagram commute:
\begin{center}
\begin{tikzpicture}
\node (M00) at (0,0) {$\mathbb{T}$};
\node (M01) at (2,0) {$\mathbb{T}$};
\node (M10) at (0,-2) {$\widehat{\mathbb{C}}$};
\node (M11) at (2,-2) {$\widehat{\mathbb{C}}$};
\draw[myarrow] (M00) to node[above]{$L$} (M01);
\draw[myarrow] (M10) to node[below]{$f$} (M11);
\draw[myarrow] (M00) to node[left]{$p$} (M10);
\draw[myarrow] (M01) to node[right]{$p$} (M11);
\end{tikzpicture}
\end{center}
\end{definition}

\begin{remark}
Suppose that $\Lambda$ is a lattice in $\mathbb{C}$ and $\mathbb{T} = \mathbb{C}/\Lambda$. Then the holomorphic maps $L \colon \mathbb{T} \rightarrow \mathbb{T}$ are precisely the maps of the form \[ L_{a, b}^{\Lambda} \colon z +\Lambda \mapsto a z +b +\Lambda \, \text{,} \quad \text{with} \quad a, b \in \mathbb{C}, \, a \Lambda \subset \Lambda \, \text{.} \] Moreover, for all $a, b \in \mathbb{C}$ such that $a \Lambda \subset \Lambda$, the map $L_{a, b}^{\Lambda} \colon \mathbb{T} \rightarrow \mathbb{T}$ has degree $\lvert a \rvert^{2}$.
\end{remark}

We distinguish two types of Latt\`{e}s maps. A rational map $f \colon \widehat{\mathbb{C}} \rightarrow \widehat{\mathbb{C}}$ of degree $d \geq 2$ is said to be a \emph{flexible} Latt\`{e}s map if there exist a torus $\mathbb{T} = \mathbb{C}/\Lambda$, with $\Lambda$ a lattice in $\mathbb{C}$, $a \in \mathbb{Z}$, $b \in \mathbb{C}$ and a holomorphic map $p \colon \mathbb{T} \rightarrow \widehat{\mathbb{C}}$ of degree $2$ such that \[ f \circ p = p \circ L_{a, b}^{\Lambda} \, \text{,} \quad \text{where} \quad L_{a, b}^{\Lambda} \colon z +\Lambda \mapsto a z +b +\Lambda \, \text{.} \] A non-flexible Latt\`{e}s map is said to be \emph{rigid}. We refer the reader to~\cite{M2006} or~\cite[Chapter~6]{S2007} for further information about Latt\`{e}s maps.

\begin{remark}
The degree of a flexible Latt\`{e}s map is the square of an integer.
\end{remark}

Given a positive squarefree integer $D$, we denote by $R_{D}$ the ring of integers of the imaginary quadratic field $\mathbb{Q}\left( i \sqrt{D} \right)$.

Latt\`{e}s maps have the following remarkable property:

\begin{proposition}[{\cite[Corollary~3.9 and Lemma~5.6]{M2006}}]
\label{proposition:specialL}
Suppose that $f \colon \widehat{\mathbb{C}} \rightarrow \widehat{\mathbb{C}}$ is a Latt\`{e}s map. Then there exists a positive squarefree integer $D$ such that the multipliers of $f$ all lie in $R_{D}$. Furthermore, the multipliers of $f$ are all integers if and only if $f$ is flexible.
\end{proposition}

We are interested in the converse of Proposition~\ref{proposition:specialPC} and Proposition~\ref{proposition:specialL}. In~\cite{M2006}, Milnor conjectured that power maps, Chebyshev maps and flexible Latt\`{e}s maps are the only rational maps whose multipliers are all integers. More generally, we may wonder whether power maps, Chebyshev maps and Latt\`{e}s maps are the only rational maps whose multipliers all lie in the ring of integers of a given imaginary quadratic field. We answer this question in the case of quadratic rational maps.

\begin{theorem}
\label{theorem:main}
Assume that $D$ is a positive squarefree integer and $f \colon \widehat{\mathbb{C}} \rightarrow \widehat{\mathbb{C}}$ is a quadratic rational map whose multiplier at each cycle with period less than or equal to $5$ lies in $R_{D}$. Then $f$ is a power map, a Chebyshev map or a Latt\`{e}s map.
\end{theorem}

In particular, this proves Milnor's conjecture for quadratic rational maps.

\begin{corollary}
Assume that $f \colon \widehat{\mathbb{C}} \rightarrow \widehat{\mathbb{C}}$ is a quadratic rational map that has only integer multipliers. Then $f$ is either a power map or a Chebyshev map.
\end{corollary}

We may even extend Milnor's question as follows:

\begin{question}
Assume that $K$ is a number field, $\mathcal{O}_{K}$ is its ring of integers and $f \colon \widehat{\mathbb{C}} \rightarrow \widehat{\mathbb{C}}$ is a rational map of degree $d \geq 2$ whose multipliers all lie in $\mathcal{O}_{K}$~-- or $K$. Is $f$ necessarily a power map, a Chebyshev map or a Latt\`{e}s map?
\end{question}

In~\cite{H2021b}, the author answered this question for certain polynomial maps. More precisely, he proved that every unicritical polynomial map of degree $d \geq 2$ that has only rational multipliers is either a power map or a Chebyshev map. He also proved that every cubic polynomial map with symmetries that has only integer multipliers is either a power map or a Chebyshev map.

In~\cite{EvS2011}, Eremenko and van~Strien studied the rational maps of degree $d \geq 2$ that have only real multipliers: they proved that, if $f \colon \widehat{\mathbb{C}} \rightarrow \widehat{\mathbb{C}}$ is such a map, then either $f$ is a Latt\`{e}s map or its Julia set $\mathcal{J}_{f}$ is contained in a circle; they also gave a description of these maps.

In Section~\ref{section:prelim}, we provide some background about the multiplier polynomials of a rational map, the moduli space of quadratic rational maps and the ring of integers of an imaginary quadratic field.

In Section~\ref{section:proof}, we prove Theorem~\ref{theorem:main}. More precisely, we determine the quadratic rational maps whose multiplier polynomials all split into linear factors in $R_{D}[\lambda]$, with $D$ a given positive squarefree integer. Using the holomorphic fixed-point formula, we are reduced to studying two one-parameter families of rational maps and finitely many other cases. We then examine the multiplier polynomials associated to these two families and to the remaining cases in order to conclude.

\begin{acknowledgments}
The author would like to thank his Ph.D. advisors, Xavier Buff and Jasmin Raissy, for their encouragements.
\end{acknowledgments}

\section{Some preliminaries}
\label{section:prelim}

We shall review here some necessary material for our proof of Theorem~\ref{theorem:main}.

\subsection{Dynatomic polynomials and multiplier polynomials}

First, we present the dynatomic polynomials and the multiplier polynomials associated to a rational map, which are related to its periodic points and its multipliers. In particular, we provide a formula to compute the multiplier polynomials of a rational map, which will be very useful in our proof of Theorem~\ref{theorem:main}. For further information about these polynomials, we refer the reader to~\cite{MS1995}, \cite{S1998} and~\cite[Chapter~4]{S2007}.

Throughout this subsection, we fix an integer $d \geq 2$, which will denote the degree of a rational map. In order to properly take the point $\infty$ into account, we identify the Riemann sphere $\widehat{\mathbb{C}}$ with the complex projective line $\mathbb{P}^{1}(\mathbb{C})$~-- defined as the quotient of $\mathbb{C}^{2} \setminus \lbrace 0 \rbrace$ by the relation of collinearity~-- by the usual biholomorphism $\iota \colon \widehat{\mathbb{C}} \rightarrow \mathbb{P}^{1}(\mathbb{C})$ and its inverse given by \[ \iota(z) = \begin{cases} [z \colon 1] & \text{if } z \in \mathbb{C}\\ [1 \colon 0] & \text{if } z = \infty \end{cases} \quad \text{and} \quad \iota^{-1}\left( [x \colon y] \right) = \begin{cases} \frac{x}{y} & \text{if } y \in \mathbb{C}^{*}\\ \infty & \text{if } y = 0 \end{cases} \, \text{.} \]

Suppose that $f \colon \mathbb{P}^{1}(\mathbb{C}) \rightarrow \mathbb{P}^{1}(\mathbb{C})$ is a rational map of degree $d$. Then there exists a homogeneous polynomial map $F \colon \mathbb{C}^{2} \rightarrow \mathbb{C}^{2}$ that does not vanish on $\mathbb{C}^{2} \setminus \lbrace 0 \rbrace$ and makes the diagram below commute, where $\pi$ denotes the canonical projection. The map $F$ is unique up to multiplication by an element of $\mathbb{C}^{*}$ and is said to be a \emph{homogeneous polynomial lift} of $f$.
\begin{center}
\begin{tikzpicture}
\node (M00) at (0,0) {$\mathbb{C}^{2} \setminus \lbrace 0 \rbrace$};
\node (M01) at (3,0) {$\mathbb{C}^{2} \setminus \lbrace 0 \rbrace$};
\node (M10) at (0,-2) {$\mathbb{P}^{1}(\mathbb{C})$};
\node (M11) at (3,-2) {$\mathbb{P}^{1}(\mathbb{C})$};
\draw[myarrow] (M00) to node[above]{$F$} (M01);
\draw[myarrow] (M10) to node[below]{$f$} (M11);
\draw[myarrow] (M00) to node[left]{$\pi$} (M10);
\draw[myarrow] (M01) to node[right]{$\pi$} (M11);
\end{tikzpicture}
\end{center}

Given a homogeneous polynomial map $F \colon \mathbb{C}^{2} \rightarrow \mathbb{C}^{2}$ of degree $d$ and $n \geq 0$, we denote by $G_{n}^{F}$ and $H_{n}^{F}$ the polynomials in $\mathbb{C}[x, y]$ defined by \[ F^{\circ n}(x, y) = \left( G_{n}^{F}(x, y), H_{n}^{F}(x, y) \right) \, \text{,} \] which are homogeneous of degree $d^{n}$.

Suppose that $f \colon \mathbb{P}^{1}(\mathbb{C}) \rightarrow \mathbb{P}^{1}(\mathbb{C})$ is a rational map of degree $d$ and $F \colon \mathbb{C}^{2} \rightarrow \mathbb{C}^{2}$ is a homogeneous polynomial lift of $f$. Then, for every $n \geq 1$, the roots in $\mathbb{P}^{1}(\mathbb{C})$ of the homogeneous polynomial \[ y G_{n}^{F}(x, y) -x H_{n}^{F}(x, y) \in \mathbb{C}[x, y] \] are precisely the periodic points for $f$ with period dividing $n$. Thus, it is natural to try to factor these polynomials in order to separate their roots according to their periods, and we obtain the result below.

For $n \geq 1$, we define \[ \nu(n) = \sum_{k \mid n} \mu\left( \frac{n}{k} \right) d^{k} \, \text{,} \] where $\mu \colon \mathbb{Z}_{\geq 1} \rightarrow \lbrace -1, 0, 1 \rbrace$ denotes the M\"{o}bius function.

\begin{proposition}[{\cite[Proposition~3.2]{MS1995}}]
Suppose that $F \colon \mathbb{C}^{2} \rightarrow \mathbb{C}^{2}$ is a homogeneous polynomial map of degree $d$ that does not vanish on $\mathbb{C}^{2} \setminus \lbrace 0 \rbrace$. Then there exists a unique sequence $\left( \Phi_{n}^{F} \right)_{n \geq 1}$ of elements of $\mathbb{C}[x, y]$ such that, for every $n \geq 1$, we have \[ y G_{n}^{F}(x, y) -x H_{n}^{F}(x, y) = \prod_{k \mid n} \Phi_{k}^{F}(x, y) \, \text{.} \] Furthermore, for every $n \geq 1$, the polynomial $\Phi_{n}^{F}$ is nonzero and homogeneous and we have \[ \deg \Phi_{n}^{F} = \begin{cases} d +1 & \text{if } n = 1\\ \nu(n) & \text{if } n \geq 2 \end{cases} \, \text{.} \]
\end{proposition}

\begin{definition}
Suppose that $F \colon \mathbb{C}^{2} \rightarrow \mathbb{C}^{2}$ is a homogeneous polynomial map of degree $d$ that does not vanish on $\mathbb{C}^{2} \setminus \lbrace 0 \rbrace$. For $n \geq 1$, the polynomial $\Phi_{n}^{F}$ is called the $n$th \emph{dynatomic polynomial} of $F$.
\end{definition}

\begin{remark}
If $F \colon \mathbb{C}^{2} \rightarrow \mathbb{C}^{2}$ is a homogeneous polynomial map of degree $d$ that does not vanish on $\mathbb{C}^{2} \setminus \lbrace 0 \rbrace$, then we have \[ \Phi_{n}^{F}(x, y) = \prod_{k \mid n} \left( y G_{k}^{F}(x, y) -x H_{k}^{F}(x, y) \right)^{\mu\left( \frac{n}{k} \right)} \] for all $n \geq 1$ by the M\"{o}bius inversion formula.
\end{remark}

The following result gives the relation between the periodic points for a rational map and the dynatomic polynomials of its homogeneous polynomial lifts.

\begin{proposition}[{\cite[Proposition~3.2]{MS1995}}]
\label{proposition:dynaRoots}
Assume that $f \colon \mathbb{P}^{1}(\mathbb{C}) \rightarrow \mathbb{P}^{1}(\mathbb{C})$ is a rational map of degree $d$, $F \colon \mathbb{C}^{2} \rightarrow \mathbb{C}^{2}$ is a homogeneous polynomial lift of $f$ and $n \geq 1$. Then $z_{0} \in \mathbb{P}^{1}(\mathbb{C})$ is a root of the polynomial $\Phi_{n}^{F}$ if and only if $z_{0}$ is either a periodic point for $f$ with period $n$ or a periodic point for $f$ with period a proper divisor $k$ of $n$ and multiplier a primitive $\frac{n}{k}$th root of unity.
\end{proposition}

Let us now present the multiplier polynomials of a rational map. Suppose that $f \colon \mathbb{P}^{1}(\mathbb{C}) \rightarrow \mathbb{P}^{1}(\mathbb{C})$ is a rational map of degree $d$ and $n \geq 1$. Informally, we want to compute the polynomial \[ \prod_{j = 1}^{r} \left( \lambda -\lambda_{j} \right) \in \mathbb{C}[\lambda] \, \text{,} \] where $\lambda_{1}, \dotsc, \lambda_{r}$ denote the multipliers of $f$ at its periodic points with period $n$. In fact, since $f$ has the same multiplier at each point of a cycle, we want to obtain the $n$th root of this polynomial. Assume that $z_{0} \in \mathbb{P}^{1}(\mathbb{C})$ is a periodic point for $f$ with period $n$ and multiplier $\lambda_{0}$ and $F \colon \mathbb{C}^{2} \rightarrow \mathbb{C}^{2}$ is a homogeneous polynomial lift of $f$. Then there exists a periodic point $\left( x_{0}, y_{0} \right) \in z_{0}$ for $F$ with period $n$, and the eigenvalues of the differential of $F^{\circ n}$ at $\left( x_{0}, y_{0} \right)$ are precisely $d^{n}$ and $\lambda_{0}$. Therefore, considering the trace of the differential of $F^{\circ n}$ at $\left( x_{0}, y_{0} \right)$, we have \[ \lambda_{0} +d^{n} = T_{n}^{F}\left( x_{0}, y_{0} \right) \, \text{,} \] where \[ T_{n}^{F} = \frac{\partial G_{n}^{F}}{\partial x} +\frac{\partial H_{n}^{F}}{\partial y} \in \mathbb{C}[x, y] \, \text{.} \] This discussion leads us to the result below (see~\cite[Chapitre~3]{H2021a}).

\begin{proposition}
\label{proposition:multDef}
Suppose that $f \colon \mathbb{P}^{1}(\mathbb{C}) \rightarrow \mathbb{P}^{1}(\mathbb{C})$ is a rational map of degree $d$, $F \colon \mathbb{C}^{2} \rightarrow \mathbb{C}^{2}$ is a homogeneous polynomial lift of $f$ and $n \geq 1$. Then there exists a unique monic polynomial $M_{n}^{f} \in \mathbb{C}[\lambda]$ such that, for every homogeneous polynomial $P \in \mathbb{C}[x, y]$ of degree $1$, we have \[ \res\left( \Phi_{n}^{F}, P \circ F^{\circ n} \right) M_{n}^{f}(\lambda)^{n} = \res\left( \Phi_{n}^{F}, \left( \lambda +d^{n} \right) P \circ F^{\circ n} -P T_{n}^{F} \right) \, \text{,} \] where $\res$ denotes the homogeneous resultant. Furthermore, $M_{n}^{f}$ depends only on $f$ and we have \[ \deg M_{n}^{f} = \begin{cases} d +1 & \text{if } n = 1\\ \frac{\nu(n)}{n} & \text{if } n \geq 2 \end{cases} \, \text{.} \]
\end{proposition}

\begin{definition}
Suppose that $f \colon \mathbb{P}^{1}(\mathbb{C}) \rightarrow \mathbb{P}^{1}(\mathbb{C})$ is a rational map of degree $d$. For $n \geq 1$, the polynomial $M_{n}^{f}$ is called the $n$th \emph{multiplier polynomial} of $f$.
\end{definition}

\begin{remark}
If $F \colon \mathbb{C}^{2} \rightarrow \mathbb{C}^{2}$ is a homogeneous polynomial map of degree $d$ that does not vanish on $\mathbb{C}^{2} \setminus \lbrace 0 \rbrace$, $n \geq 1$ and $P \in \mathbb{C}[x, y]$ is a homogeneous polynomial of degree $e \geq 0$, then we have \[ \res\left( \Phi_{n}^{F}, P \circ F^{\circ n} \right) = \epsilon_{n}^{(e)} \res\left( \Phi_{n}^{F}, P \right) \res(F)^{\frac{e \nu(n) \left( d^{n} -1 \right)}{d (d -1)}} \, \text{,} \] where $\epsilon_{n}^{(e)} \in \lbrace -1, 1 \rbrace$ equals $-1$ if and only if $n = 1$, $d$ is even and $e$ is odd.
\end{remark}

Note that, given a rational map $f \colon \mathbb{P}^{1}(\mathbb{C}) \rightarrow \mathbb{P}^{1}(\mathbb{C})$ of degree $d$, a homogeneous polynomial lift $F \colon \mathbb{C}^{2} \rightarrow \mathbb{C}^{2}$ of $f$ and $n \geq 1$, the formula in Proposition~\ref{proposition:multDef} enables us to compute the polynomial $M_{n}^{f}$ by considering a nonzero homogeneous polynomial $P \in \mathbb{C}[x, y]$ of degree $1$ that does not divide $\Phi_{n}^{F}$.

Let us now describe precisely the relation between the multiplier polynomials of a rational map and its multipliers. Given a rational map $f \colon \mathbb{P}^{1}(\mathbb{C}) \rightarrow \mathbb{P}^{1}(\mathbb{C})$ of degree $d$, a homogeneous polynomial lift $F \colon \mathbb{C}^{2} \rightarrow \mathbb{C}^{2}$ of $f$ and $n \geq 1$, we have \[ M_{n}^{f}(\lambda)^{n} = \prod_{j = 1}^{\deg \Phi_{n}^{F}} \left( \lambda -\lambda_{j} \right) \, \text{,} \] where $\lambda_{1}, \dotsc, \lambda_{\deg \Phi_{n}^{F}}$ are the multipliers of $f^{\circ n}$ at the roots of the polynomial $\Phi_{n}^{F}$ repeated according to their multiplicities. Therefore, we have the result below, which follows immediately from Proposition~\ref{proposition:dynaRoots}.

\begin{proposition}
\label{proposition:multRoots}
Assume that $f \colon \mathbb{P}^{1}(\mathbb{C}) \rightarrow \mathbb{P}^{1}(\mathbb{C})$ is a rational map of degree $d$ and $n \geq 1$. Then $\lambda_{0} \in \mathbb{C}$ is a root of the polynomial $M_{n}^{f}$ if and only if 
\begin{itemize}
\item $\lambda_{0}$ is the multiplier of $f$ at a cycle with period $n$,
\item or $\lambda_{0}$ equals $1$ and $f$ has a cycle with period a proper divisor $k$ of $n$ and multiplier a primitive $\frac{n}{k}$th root of unity.
\end{itemize}
\end{proposition}

A direct consequence of Proposition~\ref{proposition:multRoots} is the result below, which is a key point in our proof of Theorem~\ref{theorem:main}. It states that our problem comes down to determining the quadratic rational maps whose multiplier polynomials all split into linear factors in $R_{D}[\lambda]$, with $D$ a given positive squarefree integer.

\begin{corollary}
\label{corollary:multRoots}
Assume that $R$ is a subring of $\mathbb{C}$, $f \colon \mathbb{P}^{1}(\mathbb{C}) \rightarrow \mathbb{P}^{1}(\mathbb{C})$ is a rational map of degree $d$ and $n \geq 1$. Then the multipliers of $f$ at its cycles with period $n$ all lie in $R$ if and only if the polynomial $M_{n}^{f}$ splits into linear factors in $R[\lambda]$.
\end{corollary}

\subsection{The moduli space of quadratic rational maps}

We now recall certain facts about the conjugacy classes of quadratic rational maps.

Suppose that $f \colon \widehat{\mathbb{C}} \rightarrow \widehat{\mathbb{C}}$ is a quadratic rational map, and denote by $\lambda_{1}, \lambda_{2}, \lambda_{3}$ its multipliers at its fixed points repeated according to their multiplicities. If $f$ has only simple fixed points or, equivalently, if $\lambda_{j} \neq 1$ for all $j \in \lbrace 1, 2, 3 \rbrace$, then we have \[ \frac{1}{1 -\lambda_{1}} +\frac{1}{1 -\lambda_{2}} +\frac{1}{1 -\lambda_{3}} = 1 \, \text{.} \] In particular, note that $\lambda_{1} \lambda_{2} = 1$ if and only if $\lambda_{1} = \lambda_{2} = 1$.

Given a quadratic rational map $f \colon \widehat{\mathbb{C}} \rightarrow \widehat{\mathbb{C}}$, we define \[ \sigma_{1}^{f} = \lambda_{1} +\lambda_{2} +\lambda_{3} \, \text{,} \quad \sigma_{2}^{f} = \lambda_{1} \lambda_{2} +\lambda_{1} \lambda_{3} +\lambda_{2} \lambda_{3} \, \text{,} \quad \sigma_{3}^{f} = \lambda_{1} \lambda_{2} \lambda_{3} \] to be the elementary symmetric functions of the multipliers $\lambda_{1}, \lambda_{2}, \lambda_{3}$ of $f$ at its fixed points, so that \[ M_{1}^{f}(\lambda) = \lambda^{3} -\sigma_{1}^{f} \lambda^{2} +\sigma_{2}^{f} \lambda -\sigma_{3}^{f} \, \text{.} \] By the formula above that relates the multipliers of a quadratic rational map at its fixed points, for every quadratic rational map $f \colon \widehat{\mathbb{C}} \rightarrow \widehat{\mathbb{C}}$, we have \[ \sigma_{3}^{f} = \sigma_{1}^{f} -2 \, \text{.} \] In fact, we will see that this relation uniquely determines the conjugacy classes of quadratic rational maps.

We now give normal forms for the conjugacy classes of quadratic rational maps. For $a, b \in \mathbb{C}$ such that $a b \neq 1$, define \[ g_{a, b} \colon z \mapsto \frac{z (z +a)}{b z +1} \, \text{,} \] which fixes $0$ with multiplier $a$ and fixes $\infty$ with multiplier $b$. Define \[ h \colon z \mapsto z +\frac{1}{z} \, \text{,} \] which has $\infty$ as its unique fixed point. We have the following result:

\begin{proposition}[{\cite[Lemma~3.1]{M1993}}]
\label{proposition:quadConj}
Suppose that $f \colon \widehat{\mathbb{C}} \rightarrow \widehat{\mathbb{C}}$ is a quadratic rational map. If $f$ has two distinct fixed points with multipliers $a, b \in \mathbb{C}$, then we have $a b \neq 1$ and $f$ is conjugate to $g_{a, b}$. If $f$ has a unique fixed point, then $f$ is conjugate to $h$.
\end{proposition}

We will also use another normal form. For $c \in \mathbb{C}$, define \[ f_{c} \colon z \mapsto z^{2} +c \, \text{.} \] For every $c \in \mathbb{C}$, the map $f_{c}$ has $\infty$ as a superattracting fixed point. Furthermore, if $f \colon \widehat{\mathbb{C}} \rightarrow \widehat{\mathbb{C}}$ is a quadratic rational map that has a superattracting fixed point, then there exists a unique parameter $c \in \mathbb{C}$ such that $f$ is conjugate to $f_{c}$. Note that, for every $c \in \mathbb{C}$, the map $f_{c}$ is a power map if and only if $c = 0$ and is a Chebyshev map if and only if $c = -2$.

Define $\mathcal{M}_{2}(\mathbb{C})$ to be the set of conjugacy classes of quadratic rational maps. Given a quadratic rational map $f \colon \widehat{\mathbb{C}} \rightarrow \widehat{\mathbb{C}}$, denote by $[f] \in \mathcal{M}_{2}(\mathbb{C})$ its conjugacy class. We have the result below, which follows directly from Proposition~\ref{proposition:quadConj}.

\begin{corollary}[{\cite[Lemma~3.1]{M1993}}]
\label{corollary:quadConj}
The map $\Mult_{2}^{(1)} \colon \mathcal{M}_{2}(\mathbb{C}) \rightarrow \mathbb{C}^{2}$ given by \[ \Mult_{2}^{(1)}\left( [f] \right) = \left( \sigma_{1}^{f}, \sigma_{2}^{f} \right) \] is well defined and bijective. In particular, the conjugacy class of a quadratic rational map $f \colon \widehat{\mathbb{C}} \rightarrow \widehat{\mathbb{C}}$ is characterized by its multipliers $\lambda_{1}, \lambda_{2}, \lambda_{3}$ at its fixed points.
\end{corollary}

By Corollary~\ref{corollary:quadConj} and the invariance of the multiplier under conjugacy, the multiplier polynomials of $f$, with $f \colon \widehat{\mathbb{C}} \rightarrow \widehat{\mathbb{C}}$ a quadratic rational map, depend only on $\sigma_{1}^{f}$ and $\sigma_{2}^{f}$. More precisely, we have the following result:

\begin{proposition}[{\cite[Corollary~5.2]{S1998}}]
\label{proposition:quadMult}
Assume that $n \geq 1$. Then the coefficients of the polynomial $M_{n}^{f}$, with $f \colon \widehat{\mathbb{C}} \rightarrow \widehat{\mathbb{C}}$ a quadratic rational map, are polynomials in $\sigma_{1}^{f}$ and $\sigma_{2}^{f}$ with integer coefficients~-- which are independent of $f$.
\end{proposition}

\begin{remark}
If $f$ and $\overline{f}$ are quadratic rational maps with multipliers $\lambda_{1}, \lambda_{2}, \lambda_{3}$ and $\overline{\lambda_{1}}, \overline{\lambda_{2}}, \overline{\lambda_{3}}$ at their fixed points, then $M_{n}^{\overline{f}} = \overline{M_{n}^{f}}$ for all $n \geq 1$ by Proposition~\ref{proposition:quadMult}.
\end{remark}

Using the software SageMath, we can compute the first multiplier polynomials of $g_{a, b}$, with $a, b \in \mathbb{C}$ such that $a b \neq 1$. Thus, we can express the first multiplier polynomials of a quadratic rational map $f \colon \widehat{\mathbb{C}} \rightarrow \widehat{\mathbb{C}}$ in terms of $\sigma_{1}^{f}$ and $\sigma_{2}^{f}$.

\begin{example}
Suppose that $f \colon \widehat{\mathbb{C}} \rightarrow \widehat{\mathbb{C}}$ is a quadratic rational map. For simplicity, set $\sigma_{1} = \sigma_{1}^{f}$ and $\sigma_{2} = \sigma_{2}^{f}$, so that \[ M_{1}^{f}(\lambda) = \lambda^{3} -\sigma_{1} \lambda^{2} +\sigma_{2} \lambda -\left( \sigma_{1} -2 \right) \, \text{.} \] For $n \geq 1$, write \[ M_{n}^{f}(\lambda) = \lambda^{\deg M_{n}^{f}} +\sum_{j = 1}^{\deg M_{n}^{f}} (-1)^{j} \sigma_{j}^{(n)} \lambda^{\deg M_{n}^{f} -j} \, \text{.} \] Then we have 
\begin{align*}
\sigma_{1}^{(2)} & = 2 \sigma_{1} +\sigma_{2} \, \text{,}\\
\sigma_{1}^{(3)} & = \sigma_{1} \left( 2 \sigma_{1} +\sigma_{2} \right) +3 \sigma_{1} +2 \, \text{,}\\
\sigma_{2}^{(3)} & = \left( 2 \sigma_{1} +\sigma_{2} \right) \left( \sigma_{1} +\sigma_{2} \right)^{2} -\sigma_{1} \left( \sigma_{1} +2 \sigma_{2} \right) +12 \sigma_{1} +28 \, \text{,}\\
\sigma_{1}^{(4)} & = \left( 2 \sigma_{1} +\sigma_{2} \right) \sigma_{1}^{2} +\left( \sigma_{1} -\sigma_{2} \right) \left( 3 \sigma_{1} +\sigma_{2} \right) +10 \sigma_{1} \, \text{,}\\
\begin{split}
\sigma_{2}^{(4)} & = \left( 2 \sigma_{1} +\sigma_{2} \right) \sigma_{1}^{2} \left( \sigma_{1} +\sigma_{2} \right)^{2} +\left( \sigma_{1} -\sigma_{2} \right) \left( 7 \sigma_{1}^{3} +9 \sigma_{1}^{2} \sigma_{2} +5 \sigma_{1} \sigma_{2}^{2} +\sigma_{2}^{3} \right)\\
& \quad +\left( 26 \sigma_{1} -\sigma_{2} \right) \sigma_{1}^{2} +4 \sigma_{1} \left( 16 \sigma_{1} -5 \sigma_{2} \right) +4 \left( 10 \sigma_{1} -13 \sigma_{2} \right) +48 \, \text{,}
\end{split}\\
\begin{split}
\sigma_{3}^{(4)} & = \sigma_{2}^{2} \left( \sigma_{1} +\sigma_{2} \right)^{2} \left( 2 \sigma_{1} +\sigma_{2} \right)^{2} +\sigma_{1} \left( 2 \sigma_{1} +\sigma_{2} \right) \left( \sigma_{1}^{3} -2 \sigma_{1}^{2} \sigma_{2} -\sigma_{1} \sigma_{2}^{2} -2 \sigma_{2}^{3} \right)\\
& \quad +\sigma_{1} \left( 27 \sigma_{1}^{3} +30 \sigma_{1}^{2} \sigma_{2} +68 \sigma_{1} \sigma_{2}^{2} +28 \sigma_{2}^{3} \right) +4 \left( 26 \sigma_{1}^{3} +\sigma_{1}^{2} \sigma_{2} +32 \sigma_{1} \sigma_{2}^{2} +15 \sigma_{2}^{3} \right)\\
& \quad +8 \left( 37 \sigma_{1}^{2} -19 \sigma_{1} \sigma_{2} -6 \sigma_{2}^{2} \right) +32 \left( 20 \sigma_{1} +3 \sigma_{2} \right) +304 \, \text{.}
\end{split}
\end{align*}
\end{example}

Finally, let us describe the conjugacy classes of Latt\`{e}s maps of degree $2$. Suppose that $\Lambda$ is a lattice in $\mathbb{C}$, and set $\mathbb{T} = \mathbb{C}/\Lambda$. Recall that the Weierstrass's function $\wp_{\Lambda} \colon \mathbb{T} \rightarrow \widehat{\mathbb{C}}$ given by \[ \wp_{\Lambda}(z +\Lambda) = \frac{1}{z^{2}} +\sum_{w \in \Lambda \setminus \lbrace 0 \rbrace} \left( \frac{1}{(z -w)^{2}} -\frac{1}{w^{2}} \right) \] is well defined, even and holomorphic of degree $2$. Therefore, for all $a, b \in \mathbb{C}$ such that $a \Lambda \subset \Lambda$ and $2 b \in \Lambda$, there exists a unique rational map $\Lat_{a, b}^{\Lambda, 2} \colon \widehat{\mathbb{C}} \rightarrow \widehat{\mathbb{C}}$ of degree $\lvert a \rvert^{2}$ such that \[ \Lat_{a, b}^{\Lambda, 2} \circ \wp_{\Lambda} = \wp_{\Lambda} \circ L_{a, b}^{\Lambda} \, \text{,} \quad \text{where} \quad L_{a, b}^{\Lambda} \colon z +\Lambda \mapsto a z +b +\Lambda \, \text{,} \] since $L_{a, b}^{\Lambda}$ commutes with the multiplication by $-1$ in $\mathbb{T}$.

Note that certain lattices in $\mathbb{C}$ are invariant by nontrivial rotations about the origin, which gives rise to other Latt\`{e}s maps. Suppose that $\Lambda = \mathbb{Z}[i]$ and $\mathbb{T} = \mathbb{C}/\Lambda$. Then, for every $z \in \mathbb{C}$, we have \[ \wp_{\Lambda}(i z +\Lambda) = -\wp_{\Lambda}(z +\Lambda) \, \text{.} \] Therefore, for all $a, b \in \mathbb{C}$ such that $a \in \Lambda$ and $(1 +i) b \in \Lambda$, there exists a unique rational map $\Lat_{a, b}^{\Lambda, 4} \colon \widehat{\mathbb{C}} \rightarrow \widehat{\mathbb{C}}$ of degree $\lvert a \rvert^{2}$ such that \[ \Lat_{a, b}^{\Lambda, 4} \circ \wp_{\Lambda}^{2} = \wp_{\Lambda}^{2} \circ L_{a, b}^{\Lambda} \] since $L_{a, b}^{\Lambda}$ commutes with the multiplication by $i$ in $\mathbb{T}$.

We can now explicit the Latt\`{e}s maps of degree $2$ up to conjugacy.

\begin{proposition}[{\cite[Subsection~8.1]{M2006}}]
\label{proposition:quadLattes}
Assume that $f \colon \widehat{\mathbb{C}} \rightarrow \widehat{\mathbb{C}}$ is a Latt\`{e}s map of degree $2$. Then $f$ is conjugate to either $\Lat_{1 +i, 0}^{\mathbb{Z}[i], 4}$ or $\Lat_{a, b}^{\Lambda, 2}$, with \[ \Lambda \in \left\lbrace \mathbb{Z}[i], \mathbb{Z}\left[ i \sqrt{2} \right], \mathbb{Z}\left[ \frac{1 +i \sqrt{7}}{2} \right] \right\rbrace \] and \[ (a, b) \in \begin{cases} \left\lbrace (1 -i, 0), (1 +i, 0) \right\rbrace & \text{if } \Lambda = \mathbb{Z}[i]\\ \left\lbrace \left( i \sqrt{2}, 0 \right) \right\rbrace & \text{if } \Lambda = \mathbb{Z}\left[ i \sqrt{2} \right]\\ \left\lbrace \left( \frac{1 -i \sqrt{7}}{2}, 0 \right), \left( \frac{1 -i \sqrt{7}}{2}, \frac{1}{2} \right), \left( \frac{1 +i \sqrt{7}}{2}, 0 \right), \left( \frac{1 +i \sqrt{7}}{2}, \frac{1}{2} \right) \right\rbrace & \text{if } \Lambda = \mathbb{Z}\left[ \frac{1 +i \sqrt{7}}{2} \right] \end{cases} \, \text{.} \]
\end{proposition}

We can compute the multipliers of the Latt\`{e}s maps appearing in Proposition~\ref{proposition:quadLattes} at their fixed points (see Table~\ref{table:quadLattes}). Thus, we have the result below, which follows immediately from Corollary~\ref{corollary:quadConj} and Proposition~\ref{proposition:quadLattes} and gives a characterization of the Latt\`{e}s maps of degree $2$.

\begin{table}
\caption{Multipliers $\lambda_{1}, \lambda_{2}, \lambda_{3}$ of $\Lat_{a, b}^{\Lambda, n}$ at its fixed points.}
\label{table:quadLattes}
$\begin{array}{|c|c|c|c|c|}
\hline
\Lambda & n & a & b & \lambda_{1}, \lambda_{2}, \lambda_{3}\\
\hline
\hline
\mathbb{Z}[i] & 2 & 1 -i & 0 & -1 +i, -1 +i, -2 i\\
\hline
\mathbb{Z}[i] & 2 & 1 +i & 0 & -1 -i, -1 -i, 2 i\\
\hline
\mathbb{Z}[i] & 4 & 1 +i & 0 & -4, -1 -i, -1 +i\\
\hline
\mathbb{Z}\left[ i \sqrt{2} \right] & 2 & i \sqrt{2} & 0 & -2, -i \sqrt{2}, i \sqrt{2}\\
\hline
\mathbb{Z}\left[ \frac{1 +i \sqrt{7}}{2} \right] & 2 & \frac{1 -i \sqrt{7}}{2} & 0 & \frac{-3 -i \sqrt{7}}{2}, \frac{-3 -i \sqrt{7}}{2}, \frac{-1 +i \sqrt{7}}{2}\\
\hline
\mathbb{Z}\left[ \frac{1 +i \sqrt{7}}{2} \right] & 2 & \frac{1 -i \sqrt{7}}{2} & \frac{1}{2} & \frac{-1 +i \sqrt{7}}{2}, \frac{-1 +i \sqrt{7}}{2}, \frac{1 -i \sqrt{7}}{2}\\
\hline
\mathbb{Z}\left[ \frac{1 +i \sqrt{7}}{2} \right] & 2 & \frac{1 +i \sqrt{7}}{2} & 0 & \frac{-3 +i \sqrt{7}}{2}, \frac{-3 +i \sqrt{7}}{2}, \frac{-1 -i \sqrt{7}}{2}\\
\hline
\mathbb{Z}\left[ \frac{1 +i \sqrt{7}}{2} \right] & 2 & \frac{1 +i \sqrt{7}}{2} & \frac{1}{2} & \frac{-1 -i \sqrt{7}}{2}, \frac{-1 -i \sqrt{7}}{2}, \frac{1 +i \sqrt{7}}{2}\\
\hline
\end{array}$
\end{table}

\begin{corollary}
\label{corollary:quadLattes}
Assume that $f \colon \widehat{\mathbb{C}} \rightarrow \widehat{\mathbb{C}}$ is a quadratic rational map. Then $f$ is a Latt\`{e}s map if and only if its multipliers at its fixed points are 
\begin{itemize}
\item either $-4$, $-1 -i$ and $-1 +i$,
\item or $-1 -i$, $-1 -i$ and $2 i$,
\item or $-1 +i$, $-1 +i$ and $-2 i$,
\item or $-2$, $-i \sqrt{2}$ and $i \sqrt{2}$,
\item or $\frac{-3 -i \sqrt{7}}{2}$, $\frac{-3 -i \sqrt{7}}{2}$ and $\frac{-1 +i \sqrt{7}}{2}$,
\item or $\frac{-3 +i \sqrt{7}}{2}$, $\frac{-3 +i \sqrt{7}}{2}$ and $\frac{-1 -i \sqrt{7}}{2}$,
\item or $\frac{-1 -i \sqrt{7}}{2}$, $\frac{-1 -i \sqrt{7}}{2}$ and $\frac{1 +i \sqrt{7}}{2}$,
\item or $\frac{-1 +i \sqrt{7}}{2}$, $\frac{-1 +i \sqrt{7}}{2}$ and $\frac{1 -i \sqrt{7}}{2}$.
\end{itemize}
\end{corollary}

\subsection{The ring of integers of an imaginary quadratic field}

Finally, we recall here some properties of the ring of integers $R_{D}$ of the imaginary quadratic field $\mathbb{Q}\left( i \sqrt{D} \right)$, with $D$ a positive squarefree integer.

Assume that $D$ is a positive squarefree integer. Then we have \[ R_{D} = \mathbb{Z}\left[ \alpha_{D} \right] \, \text{,} \quad \text{where} \quad \alpha_{D} = \begin{cases} i \sqrt{D} & \text{if } D \equiv 1, 2 \pmod{4}\\ \frac{1 +i \sqrt{D}}{2} & \text{if } D \equiv 3 \pmod{4} \end{cases} \, \text{.} \]

The ring $R_{D}$ is an integrally closed domain. In particular, if $a, b \in R_{D}$ are such that $a b^{2}$ is a square in $R_{D}$, then $a$ is a square in $R_{D}$ or $b$ is zero. This property will be useful in our proof of Theorem~\ref{theorem:main}.

The elements of $R_{D}$ form a lattice in $\mathbb{C}$. Let us describe the intersections of $R_{D}$ with the Euclidean disks centered at the origin. We denote by $N \colon \mathbb{C} \rightarrow \mathbb{R}_{\geq 0}$ the map given by $N(z) = \lvert z \rvert^{2}$, which is multiplicative and agrees with the norm of the extension $\mathbb{Q}\left( i \sqrt{D} \right)/\mathbb{Q}$.

Suppose that $D \equiv 1, 2 \pmod{4}$. For all $x, y \in \mathbb{Z}$, we have \[ N\left( x +y \alpha_{D} \right) = x^{2} +D y^{2} \in \mathbb{Z}_{\geq 0} \, \text{.} \] Therefore, for every $B \geq 0$, we have \[ \left\lbrace z \in R_{D} : N(z) \leq B \right\rbrace \subset \left\lbrace x +y \alpha_{D} : x, y \in \mathbb{Z}, \, \lvert x \rvert \leq \sqrt{B}, \, \lvert y \rvert \leq \sqrt{\frac{B}{D}} \right\rbrace \, \text{,} \] and in particular \[ \left\lbrace z \in R_{D} : N(z) \leq B \right\rbrace \subset \mathbb{Z} \quad \text{if} \quad B < D \, \text{.} \]

Suppose that $D \equiv 3 \pmod{4}$. For all $x, y \in \mathbb{Z}$, we have \[ N\left( x +y \alpha_{D} \right) = \left( x +\frac{1}{2} y \right)^{2} +\frac{D}{4} y^{2} = x^{2} +x y +\frac{D +1}{4} y^{2} \in \mathbb{Z}_{\geq 0} \, \text{.} \] Therefore, for every $B \geq 0$, we have \[ \left\lbrace z \in R_{D} : N(z) \leq B \right\rbrace \subset \left\lbrace x +y \alpha_{D} : x, y \in \mathbb{Z}, \, \lvert x \rvert \leq \sqrt{B} +\sqrt{\frac{B}{D}}, \, \lvert y \rvert \leq 2 \sqrt{\frac{B}{D}} \right\rbrace \, \text{,} \] and in particular \[ \left\lbrace z \in R_{D} : N(z) \leq B \right\rbrace \subset \mathbb{Z} \quad \text{if} \quad 4 B < D \, \text{.} \]

Thus, the set of all imaginary quadratic integers is a discrete subset of $\mathbb{C}$ and, for every $B \geq 0$, we can determine the pairs $(D, z) \in \mathbb{Z} \times \mathbb{C}$ such that $D$ is a positive squarefree integer, $z \in R_{D}$ and $N(z) \leq B$.

\section{Proof of the result}
\label{section:proof}

We shall prove here Theorem~\ref{theorem:main}. It follows directly from the three lemmas below.

\begin{lemma}
\label{lemma:proofCases}
Assume that $D$ is a positive squarefree integer and $f \colon \widehat{\mathbb{C}} \rightarrow \widehat{\mathbb{C}}$ is a quadratic rational map that has no superattracting or multiple fixed point and whose multiplier at each cycle with period less than or equal to $5$ lies in $R_{D}$. Then $f$ is either a power map or a Latt\`{e}s map.
\end{lemma}

\begin{proof}
Denote by $\lambda_{1}, \lambda_{2}, \lambda_{3}$ the multipliers of $f$ at its fixed points, which belong to $R_{D} \setminus \lbrace 0, 1 \rbrace$ by hypothesis. Then $1 -\lambda_{j}$ lies in $R_{D} \setminus \lbrace 0, 1 \rbrace$ for all $j \in \lbrace 1, 2, 3 \rbrace$ and we have \[ \frac{1}{1 -\lambda_{1}} +\frac{1}{1 -\lambda_{2}} +\frac{1}{1 -\lambda_{3}} = 1 \, \text{.} \] If $\mu_{1}, \mu_{2}, \mu_{3}$ are elements of $R_{D} \setminus \lbrace 0, 1 \rbrace$ that satisfy \[ \frac{1}{\mu_{1}} +\frac{1}{\mu_{2}} +\frac{1}{\mu_{3}} = 1 \quad \text{and} \quad \Re\left( \frac{1}{\mu_{1}} \right) \leq \Re\left( \frac{1}{\mu_{2}} \right) \leq \Re\left( \frac{1}{\mu_{3}} \right) \, \text{,} \] then we have \[ \Re\left( \frac{1}{\mu_{3}} \right) \geq \frac{1}{3} \, \text{,} \quad \Re\left( \frac{1}{\mu_{2}} \right) \geq \frac{1}{2} \left( 1 -\Re\left( \frac{1}{\mu_{3}} \right) \right) \geq \frac{1}{4} \quad \text{and} \quad \frac{1}{\mu_{1}} = 1 -\frac{1}{\mu_{2}} -\frac{1}{\mu_{3}} \] since $\Re\left( \frac{1}{z} \right) \leq \frac{1}{2}$ for all $z \in R_{D} \setminus \lbrace 0, 1 \rbrace$. Therefore, there are only finitely many unordered triples $\mu_{1}, \mu_{2}, \mu_{3}$ of elements of $R_{D} \setminus \lbrace 0, 1 \rbrace$ such that $\frac{1}{\mu_{1}} +\frac{1}{\mu_{2}} +\frac{1}{\mu_{3}} = 1$. If $D = 1$, then there are exactly $23$ such unordered triples (see Figure~\ref{figure:proofCases1}); if $D = 2$, there are $9$; if $D = 3$, there are $27$ (see Figure~\ref{figure:proofCases3}); if $D = 7$, there are $14$; if $D = 11$, there are $3$; if $D = 15$, there are $5$. In the other cases, $2$, $3$ and $4$ are the only elements $z \in R_{D} \setminus \lbrace 0, 1 \rbrace$ such that $\Re\left( \frac{1}{z} \right) \geq \frac{1}{4}$, and it follows that the only triples $\left( \mu_{1}, \mu_{2}, \mu_{3} \right)$ of elements of $R_{D} \setminus \lbrace 0, 1 \rbrace$ such that $\frac{1}{\mu_{1}} +\frac{1}{\mu_{2}} +\frac{1}{\mu_{3}} = 1$ are $(2, 3, 6)$, $(2, 4, 4)$ and $(3, 3, 3)$ up to permutation (see Figure~\ref{figure:proofCases5} and Figure~\ref{figure:proofCases19}). Thus, there are only finitely many possible values for the triple $\left( \lambda_{1}, \lambda_{2}, \lambda_{3} \right)$, and these are $(-5, -2, -1)$, $(-3, -3, -1)$ and $(-2, -2, -2)$ up to permutation if $D$ is different from $1$, $2$, $3$, $7$, $11$ and $15$. If $\lambda_{1}, \lambda_{2}, \lambda_{3}$ equal $-5, -2, -1$, then we have \[ M_{4}^{f}(\lambda) = \lambda^{3} -159 \lambda^{2} +7419 \lambda -84221 \, \text{,} \] which does not split into linear factors in $R_{D}[\lambda]$ since it is irreducible over $\mathbb{Q}$ of degree $3$ and $R_{D}$ is contained in an extension of $\mathbb{Q}$ of degree $2$. If $\lambda_{1}, \lambda_{2}, \lambda_{3}$ equal $-3, -3, -1$, then we have \[ M_{5}^{f}(\lambda) = \left( \lambda^{3} +267 \lambda^{2} +20871 \lambda +414157 \right)^{2} \, \text{,} \] which does not split into linear factors in $R_{D}[\lambda]$ either since it is the square of an irreducible polynomial over $\mathbb{Q}$ of degree $3$ and $R_{D}$ is contained in an extension of $\mathbb{Q}$ of degree $2$. Therefore, since the polynomials $M_{n}^{f}$, with $n \in \lbrace 3, 4, 5 \rbrace$, split into linear factors in $R_{D}[\lambda]$ by Corollary~\ref{corollary:multRoots}, we have \[ \left( \lambda_{1}, \lambda_{2}, \lambda_{3} \right) \in \begin{cases} \begin{gathered} {\textstyle \lbrace (-4, -1 -i, -1 +i), (-2, -2, -2),}\\ {\textstyle (-1 -i, -1 -i, 2 i), (-1 +i, -1 +i, -2 i) \rbrace} \end{gathered} & \text{if } D = 1\\ \left\lbrace (-2, -2, -2), \left( -2, -i \sqrt{2}, i \sqrt{2} \right) \right\rbrace & \text{if } D = 2\\ \begin{gathered} {\textstyle \Bigl\lbrace (-2, -2, -2), \left( \frac{-3 -i \sqrt{7}}{2}, \frac{-3 -i \sqrt{7}}{2}, \frac{-1 +i \sqrt{7}}{2} \right),}\\ {\textstyle \left( \frac{-3 +i \sqrt{7}}{2}, \frac{-3 +i \sqrt{7}}{2}, \frac{-1 -i \sqrt{7}}{2} \right), \left( \frac{-1 -i \sqrt{7}}{2}, \frac{-1 -i \sqrt{7}}{2}, \frac{1 +i \sqrt{7}}{2} \right),}\\ {\textstyle \left( \frac{-1 +i \sqrt{7}}{2}, \frac{-1 +i \sqrt{7}}{2}, \frac{1 -i \sqrt{7}}{2} \right) \Bigr\rbrace} \end{gathered} & \text{if } D = 7\\ \left\lbrace (-2, -2, -2) \right\rbrace & \text{otherwise} \end{cases} \] up to permutation (see Table~\ref{table:proofCases3}, Table~\ref{table:proofCases4} and Table~\ref{table:proofCases5}). If $\lambda_{1}, \lambda_{2}, \lambda_{3}$ equal $-2, -2, -2$, then $f$ is conjugate to $z \mapsto \frac{1}{z^{2}}$; in the other cases, $f$ is a Latt\`{e}s map by Corollary~\ref{corollary:quadLattes}. Thus, the lemma is proved.
\end{proof}

\begin{figure}
\framebox{\includegraphics{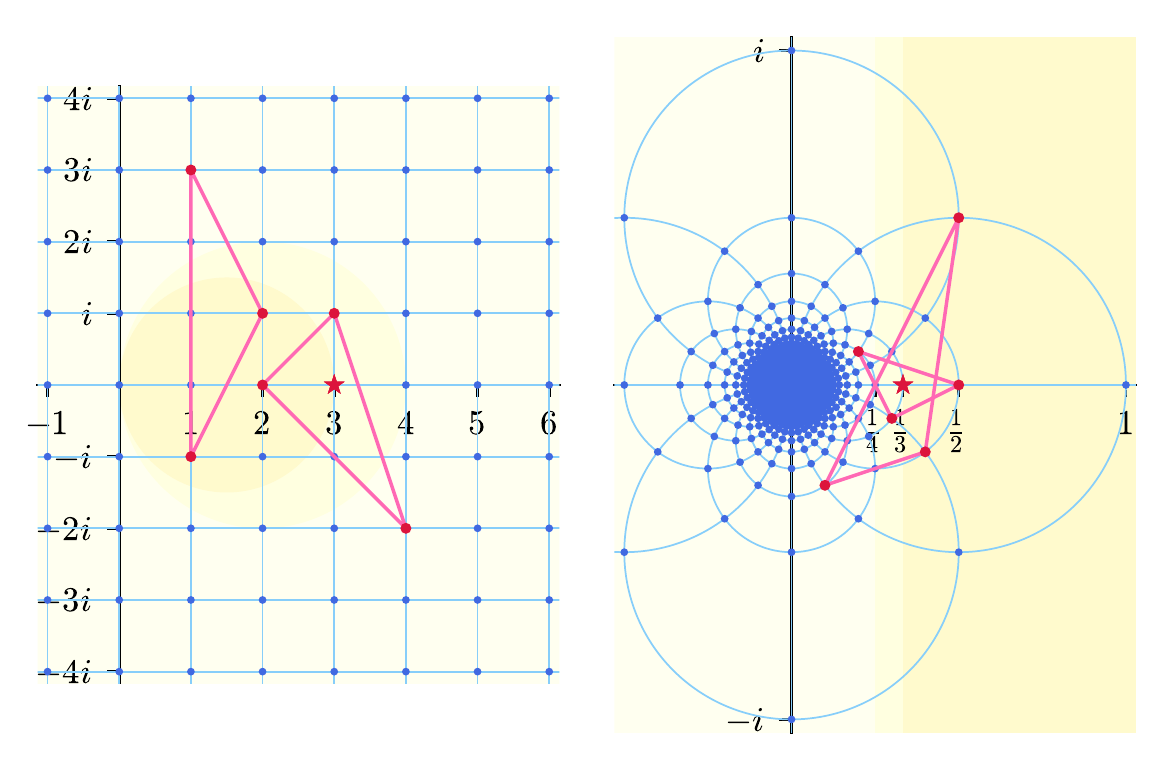}}
\caption{Left: The lattice $R_{1}$ and $3$ of the $23$ unordered triples $\mu_{1}, \mu_{2}, \mu_{3}$ of elements of $R_{1} \setminus \lbrace 0, 1 \rbrace$ such that $\frac{1}{\mu_{1}} +\frac{1}{\mu_{2}} +\frac{1}{\mu_{3}} = 1$. Right: The inversion of $R_{1}$ and of these triples. If $\mu_{1}, \mu_{2}, \mu_{3}$ is such a triple, then, up to relabeling, we have $\Re\left( \frac{1}{\mu_{3}} \right) \geq \frac{1}{3}$, $\Re\left( \frac{1}{\mu_{2}} \right) \geq \frac{1}{4}$ and $\frac{1}{3}$ is the centroid of the triangle with vertices $\frac{1}{\mu_{1}}, \frac{1}{\mu_{2}}, \frac{1}{\mu_{3}}$.}
\label{figure:proofCases1}
\end{figure}

\begin{figure}
\framebox{\includegraphics{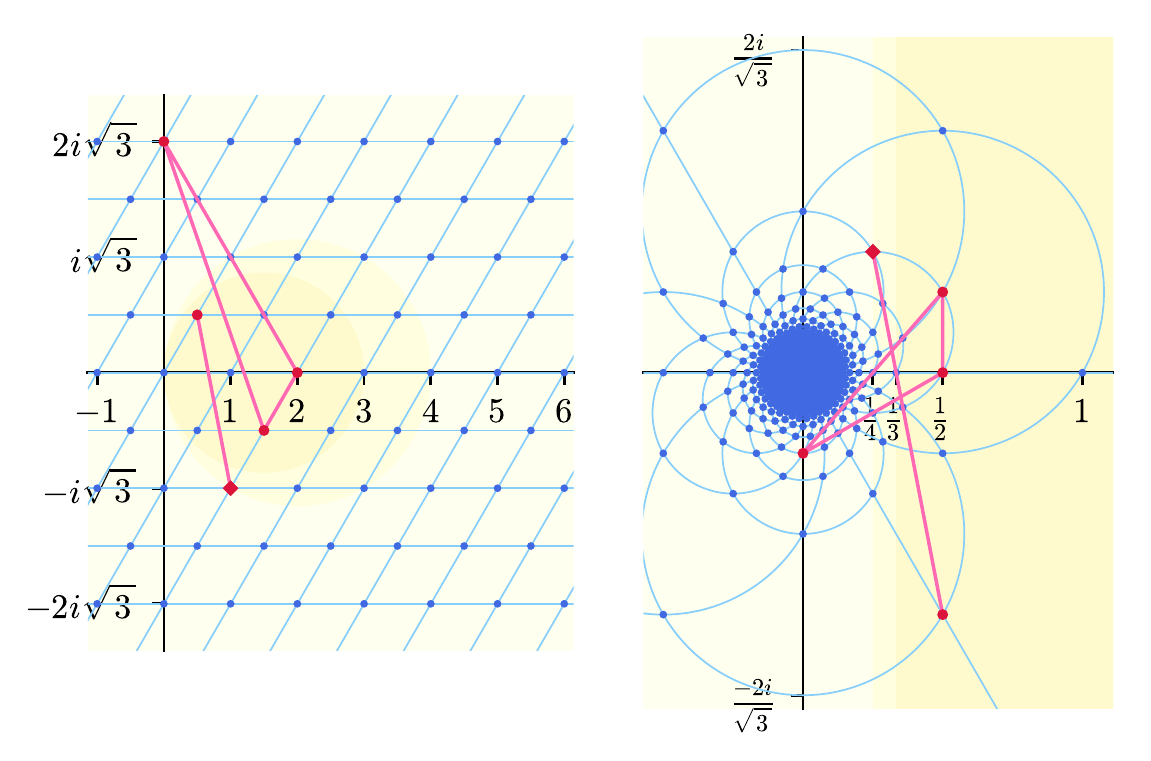}}
\caption{Left: The lattice $R_{3}$ and $2$ of the $27$ unordered triples $\mu_{1}, \mu_{2}, \mu_{3}$ of elements of $R_{3} \setminus \lbrace 0, 1 \rbrace$ such that $\frac{1}{\mu_{1}} +\frac{1}{\mu_{2}} +\frac{1}{\mu_{3}} = 1$. Right: The inversion of $R_{3}$ and of these triples.}
\label{figure:proofCases3}
\end{figure}

\begin{figure}
\framebox{\includegraphics{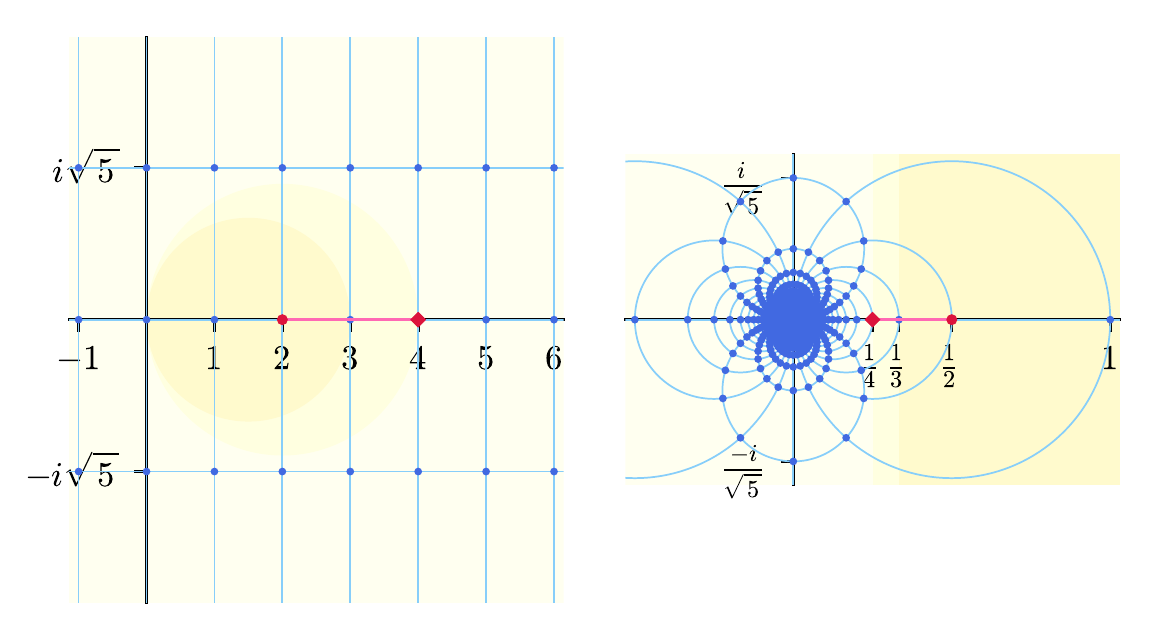}}
\caption{Left: The lattice $R_{5}$ and $1$ of the $3$ unordered triples $\mu_{1}, \mu_{2}, \mu_{3}$ of elements of $R_{5} \setminus \lbrace 0, 1 \rbrace$ such that $\frac{1}{\mu_{1}} +\frac{1}{\mu_{2}} +\frac{1}{\mu_{3}} = 1$. Right: The inversion of $R_{5}$ and of this triple. The only elements $z \in R_{5} \setminus \lbrace 0, 1 \rbrace$ such that $\Re\left( \frac{1}{z} \right) \geq \frac{1}{4}$ are $2$, $3$ and $4$.}
\label{figure:proofCases5}
\end{figure}

\begin{figure}
\framebox{\includegraphics{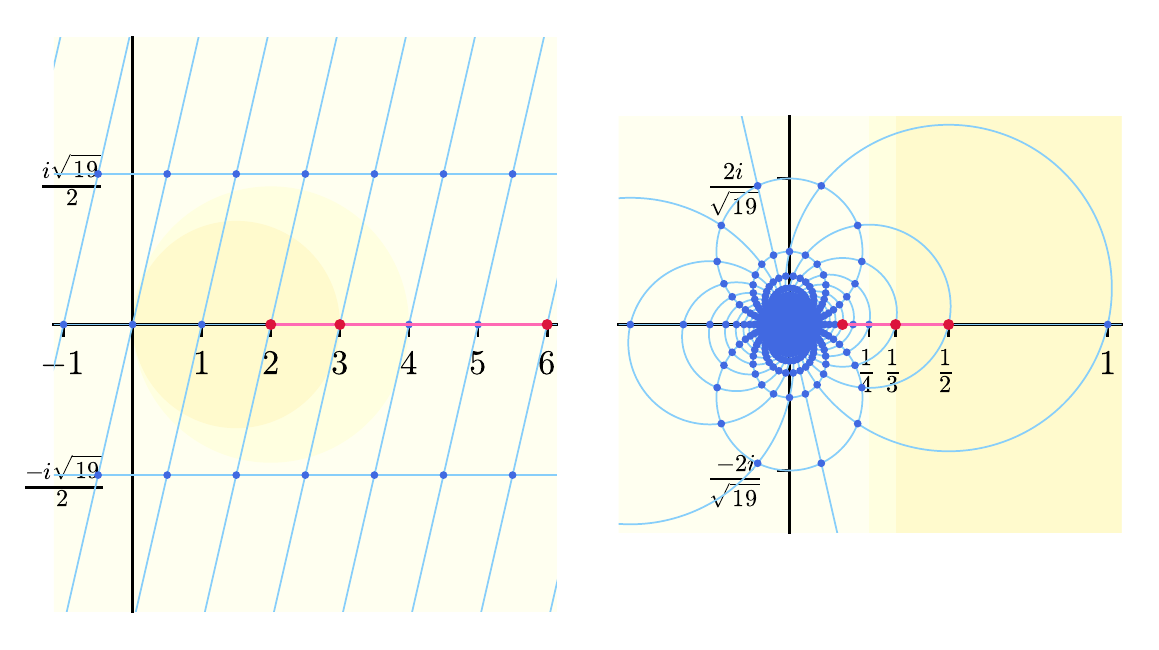}}
\caption{Left: The lattice $R_{19}$ and $1$ of the $3$ unordered triples $\mu_{1}, \mu_{2}, \mu_{3}$ of elements of $R_{19} \setminus \lbrace 0, 1 \rbrace$ such that $\frac{1}{\mu_{1}} +\frac{1}{\mu_{2}} +\frac{1}{\mu_{3}} = 1$. Right: The inversion of $R_{19}$ and of this triple.}
\label{figure:proofCases19}
\end{figure}

\begin{table}
\caption{Decomposition of $M_{3}^{f}$ into irreducible factors in $R_{D}[\lambda]$ for all the unordered triples $\lambda_{1}, \lambda_{2}, \lambda_{3}$ of elements of $\mathbb{C} \setminus \lbrace 0, 1 \rbrace$~-- up to complex conjugation~-- such that $M_{1}^{f}$ and $M_{2}^{f}$ split into linear factors in $R_{D}[\lambda]$ but $M_{3}^{f}$ does not.}
\label{table:proofCases3}
$\begin{array}{|c|c|c|}
\hline
D & \lambda_{1}, \lambda_{2}, \lambda_{3} & \text{Factorization of } M_{3}^{f} \text{ in } R_{D}[\lambda]\\
\hline
\hline
1 & -3 -2 i, -2 +i, -1 & \lambda^{2} +(22 +4 i) \lambda +121 +40 i\\
\hline
1 & -1 -4 i, -1, -1 +i & \lambda^{2} +(10 +12 i) \lambda +5 +48 i\\
\hline
1 & -1 -i, -3 i, i & \lambda^{2} +(12 +2 i) \lambda +15 -28 i\\
\hline
1 & -1, -i, 1 +2 i & \lambda^{2} +(-2 -4 i) \lambda +25 +8 i\\
\hline
2 & -1 -2 i \sqrt{2}, -1, -1 +i \sqrt{2} & \lambda^{2} +\left( 10 +4 i \sqrt{2} \right) \lambda +33 +16 i \sqrt{2}\\
\hline
3 & -3 -2 i \sqrt{3}, \frac{-3 +i \sqrt{3}}{2}, -1 & \lambda^{2} +\left( 20 +6 i \sqrt{3} \right) \lambda +79 +54 i \sqrt{3}\\
\hline
3 & -2 -i \sqrt{3}, -2 +i \sqrt{3}, -1 & \lambda^{2} +18 \lambda +89\\
\hline
3 & -1, -i \sqrt{3}, i \sqrt{3} & \lambda^{2} +2 \lambda +25\\
\hline
7 & \frac{-5 -i \sqrt{7}}{2}, \frac{-5 +i \sqrt{7}}{2}, -1 & \lambda^{2} +22 \lambda +125\\
\hline
7 & -2 -i \sqrt{7}, \frac{-3 +i \sqrt{7}}{2}, -1 & \lambda^{2} +\left( 16 +2 i \sqrt{7} \right) \lambda +67 +14 i \sqrt{7}\\
\hline
7 & -1, \frac{-1 -i \sqrt{7}}{2}, i \sqrt{7} & \lambda^{2} +\left( 4 -2 i \sqrt{7} \right) \lambda +19 -2 i \sqrt{7}\\
\hline
15 & \frac{-3 -i \sqrt{15}}{2}, \frac{-3 +i \sqrt{15}}{2}, -1 & \lambda^{2} +14 \lambda +61\\
\hline
15 & -1, \frac{-1 -i \sqrt{15}}{2}, \frac{-1 +i \sqrt{15}}{2} & \lambda^{2} +6 \lambda +29\\
\hline
\end{array}$
\end{table}

\begin{table}
\caption{Decomposition of $M_{4}^{f}$ into irreducible factors in $R_{D}[\lambda]$ for all the unordered triples $\lambda_{1}, \lambda_{2}, \lambda_{3}$ of elements of $\mathbb{C} \setminus \lbrace 0, 1 \rbrace$ other than $-5, -2, -1$~-- up to complex conjugation~-- such that $M_{n}^{f}$ splits into linear factors in $R_{D}[\lambda]$, with $n \in \lbrace 1, 2, 3 \rbrace$, but $M_{4}^{f}$ does not.}
\label{table:proofCases4}
$\begin{array}{|c|c|c|}
\hline
D & \lambda_{1}, \lambda_{2}, \lambda_{3} & \text{Factorization of } M_{4}^{f} \text{ in } R_{D}[\lambda]\\
\hline
\hline
1 & -2 -i, -2 i, i & (\lambda -1) \left( \lambda^{2} +(6 +12 i) \lambda +41 +60 i \right)\\
\hline
1 & -1 -2 i, -1, -1 +2 i & (\lambda -11) \left( \lambda^{2} +12 \lambda +211 \right)\\
\hline
2 & -2, -1 -i \sqrt{2}, -1 +i \sqrt{2} & (\lambda -1) \left( \lambda^{2} +2 \lambda +37 \right)\\
\hline
3 & \frac{-7 -i \sqrt{3}}{2}, -1 -i \sqrt{3}, \frac{-1 +i \sqrt{3}}{2} & \begin{gathered} {\textstyle \lambda^{3} +\frac{99 -3 i \sqrt{3}}{2} \lambda^{2}}\\ {\textstyle +\frac{1449 +9 i \sqrt{3}}{2} \lambda +4267 +768 i \sqrt{3}} \end{gathered}\\
\hline
3 & -2 -2 i \sqrt{3}, \frac{-3 -i \sqrt{3}}{2}, \frac{-1 +i \sqrt{3}}{2} & \begin{gathered} {\textstyle \left( \lambda +\frac{1 +7 i \sqrt{3}}{2} \right)}\\ {\textstyle \left( \lambda^{2} +\left( 5 -i \sqrt{3} \right) \lambda +\frac{95 -17 i \sqrt{3}}{2} \right)} \end{gathered}\\
\hline
3 & -2 -i \sqrt{3}, \frac{-1 -i \sqrt{3}}{2}, i \sqrt{3} & \begin{gathered} {\textstyle \left( \lambda +8 +5 i \sqrt{3} \right)}\\ {\textstyle \left( \lambda^{2} +\left( -19 -i \sqrt{3} \right) \lambda -62 +65 i \sqrt{3} \right)} \end{gathered}\\
\hline
3 & -2, \frac{-1 -3 i \sqrt{3}}{2}, \frac{-1 +i \sqrt{3}}{2} & \begin{gathered} {\textstyle \lambda^{3} +\frac{39 -7 i \sqrt{3}}{2} \lambda^{2}}\\ {\textstyle +\frac{261 -19 i \sqrt{3}}{2} \lambda +449 -302 i \sqrt{3}} \end{gathered}\\
\hline
3 & -1, \frac{-1 -i \sqrt{3}}{2}, 1 +2 i \sqrt{3} & \begin{gathered} {\textstyle \lambda^{3} +\left( 33 -12 i \sqrt{3} \right) \lambda^{2}}\\ {\textstyle +\left( -297 -132 i \sqrt{3} \right) \lambda +103 +1392 i \sqrt{3}} \end{gathered}\\
\hline
3 & \frac{-1 -i \sqrt{3}}{2}, \frac{1 +i \sqrt{3}}{2}, 1 -i \sqrt{3} & \begin{gathered} {\textstyle \lambda^{3} +\frac{27 +27 i \sqrt{3}}{2} \lambda^{2}}\\ {\textstyle +\frac{-423 +39 i \sqrt{3}}{2} \lambda +883 +624 i \sqrt{3}} \end{gathered}\\
\hline
7 & -3, \frac{-1 -i \sqrt{7}}{2}, \frac{-1 +i \sqrt{7}}{2} & \lambda^{3} +25 \lambda^{2} +187 \lambda +587\\
\hline
7 & -1, \frac{1 -i \sqrt{7}}{2}, \frac{1 +i \sqrt{7}}{2} & \lambda^{3} +\lambda^{2} -5 \lambda -413\\
\hline
\end{array}$
\end{table}

\begin{table}
\caption{Decomposition of $M_{5}^{f}$ into irreducible factors in $R_{D}[\lambda]$ for all the unordered triples $\lambda_{1}, \lambda_{2}, \lambda_{3}$ of elements of $\mathbb{C} \setminus \lbrace 0, 1 \rbrace$ other than $-3, -3, -1$~-- up to complex conjugation~-- such that $M_{n}^{f}$ splits into linear factors in $R_{D}[\lambda]$, with $n \in \lbrace 1, \dotsc, 4 \rbrace$, but $M_{5}^{f}$ does not.}
\label{table:proofCases5}
$\begin{array}{|c|c|c|}
\hline
D & \lambda_{1}, \lambda_{2}, \lambda_{3} & \text{Factorization of } M_{5}^{f} \text{ in } R_{D}[\lambda]\\
\hline
\hline
1 & -2 -i, -2 -i, -1 +i & \begin{gathered} {\textstyle \bigl( \lambda^{3} +(10 +23 i) \lambda^{2}}\\ {\textstyle +(33 +188 i) \lambda +758 +1703 i \bigr)^{2}} \end{gathered}\\
\hline
1 & -1 -2 i, -1 -2 i, i & \begin{gathered} {\textstyle \bigl( \lambda^{3} +(-5 +32 i) \lambda^{2}}\\ {\textstyle +(-633 -640 i) \lambda +605 -11584 i \bigr)^{2}} \end{gathered}\\
\hline
1 & -i, -i, 1 +i & \begin{gathered} {\textstyle \bigl( \lambda^{3} +(4 +31 i) \lambda^{2}}\\ {\textstyle +(-171 -176 i) \lambda -700 +1699 i \bigr)^{2}} \end{gathered}\\
\hline
2 & -1 -i \sqrt{2}, -1 -i \sqrt{2}, i \sqrt{2} & \begin{gathered} {\textstyle \bigl( \lambda^{3} +\left( 3 +3 i \sqrt{2} \right) \lambda^{2}}\\ {\textstyle +\left( -27 -42 i \sqrt{2} \right) \lambda +3 -343 i \sqrt{2} \bigr)^{2}} \end{gathered}\\
\hline
3 & -2 -i \sqrt{3}, -2 -i \sqrt{3}, \frac{-1 +i \sqrt{3}}{2} & \begin{gathered} {\textstyle \bigl( \lambda^{3} +\left( -12 -21 i \sqrt{3} \right) \lambda^{2}}\\ {\textstyle +\left( -573 +36 i \sqrt{3} \right) \lambda -8380 +2709 i \sqrt{3} \bigr)^{2}} \end{gathered}\\
\hline
3 & \frac{-3 -i \sqrt{3}}{2}, \frac{-3 -i \sqrt{3}}{2}, -1 +i \sqrt{3} & \begin{gathered} {\textstyle \Bigl( \lambda^{3} +\frac{15 -5 i \sqrt{3}}{2} \lambda^{2}}\\ {\textstyle +\frac{-87 -169 i \sqrt{3}}{2} \lambda -320 -709 i \sqrt{3} \Bigr)^{2}} \end{gathered}\\
\hline
3 & \frac{-1 -i \sqrt{3}}{2}, \frac{-1 -i \sqrt{3}}{2}, 1 +i \sqrt{3} & \begin{gathered} {\textstyle \Bigl( \lambda^{3} +\frac{3 -3 i \sqrt{3}}{2} \lambda^{2}}\\ {\textstyle +\frac{-147 +45 i \sqrt{3}}{2} \lambda -577 -720 i \sqrt{3} \Bigr)^{2}} \end{gathered}\\
\hline
3 & -i \sqrt{3}, -i \sqrt{3}, \frac{1 +i \sqrt{3}}{2} & \begin{gathered} {\textstyle \bigl( \lambda^{3} +\left( 42 -29 i \sqrt{3} \right) \lambda^{2}}\\ {\textstyle +\left( -1329 -232 i \sqrt{3} \right) \lambda -7742 +4897 i \sqrt{3} \bigr)^{2}} \end{gathered}\\
\hline
\end{array}$
\end{table}

By Lemma~\ref{lemma:proofCases}, we are reduced to studying the quadratic rational maps that have a superattracting or multiple fixed point.

\begin{lemma}
\label{lemma:proofSuper}
Assume that $D$ is a positive squarefree integer and $f \colon \widehat{\mathbb{C}} \rightarrow \widehat{\mathbb{C}}$ is a quadratic rational map that has a superattracting fixed point and whose multiplier at each cycle with period less than or equal to $4$ lies in $R_{D}$. Then $f$ is either a power map or a Chebyshev map.
\end{lemma}

\begin{proof}
There exists a parameter $c \in \mathbb{C}$ such that $f$ is conjugate to $f_{c} \colon z \mapsto z^{2} +c$. Let us prove that $c \in \lbrace -2, 0 \rbrace$. By Corollary~\ref{corollary:multRoots}, the polynomials \[ M_{1}^{f_{c}}(\lambda) = \lambda^{3} -2 \lambda^{2} +4 c \lambda \quad \text{and} \quad M_{3}^{f_{c}}(\lambda) = \lambda^{2} +(-8 c -16) \lambda +64 c^{3} +128 c^{2} +64 c +64 \] split into linear factors in $R_{D}[\lambda]$, and hence $4 c$ lies in $R_{D}$ and the discriminants \[ \disc M_{1}^{f_{c}} = -2^{2} (4 c -1) (4 c)^{2} \quad \text{and} \quad \disc M_{3}^{f_{c}} = -2^{2} (4 c +7) (4 c)^{2} \] are squares in $R_{D}$. Therefore, we have $c = 0$ or there exist $a, b \in R_{D}$ such that \[ -(4 c -1) = a^{2} \quad \text{and} \quad -(4 c +7) = b^{2} \, \text{.} \] In the latter case, we have $(a -b) (a +b) = 8$, which yields \[ a = \frac{(a -b)^{2} +8}{2 (a -b)} \in R_{D} \cap \left\lbrace \frac{e^{2} +8}{2 e} : e \in R_{D} \text{ and } N(e) \text{ divides } 64 \right\rbrace \, \text{,} \] and hence \[ a \in \begin{cases} \lbrace -3, -2, -i, i, 2, 3 \rbrace & \text{if } D = 1\\ \lbrace -3, 0, 3 \rbrace & \text{if } D = 2\\ \left\lbrace -3, \frac{-3 -i \sqrt{3}}{2}, \frac{-3 +i \sqrt{3}}{2}, \frac{3 -i \sqrt{3}}{2}, \frac{3 +i \sqrt{3}}{2}, 3 \right\rbrace & \text{if } D = 3\\ \lbrace -3, -1, 1, 3 \rbrace & \text{if } D = 7\\ \lbrace -3, 3 \rbrace & \text{otherwise} \end{cases} \, \text{.} \] Therefore, in the latter case, we have \[ c = \frac{1 -a^{2}}{4} \in \begin{cases} \left\lbrace -2, \frac{-3}{4}, \frac{1}{2} \right\rbrace & \text{if } D = 1\\ \left\lbrace -2, \frac{1}{4} \right\rbrace & \text{if } D = 2\\ \left\lbrace -2, \frac{-1 -3 i \sqrt{3}}{8}, \frac{-1 +3 i \sqrt{3}}{8} \right\rbrace & \text{if } D = 3\\ \lbrace -2, 0 \rbrace & \text{if } D = 7\\ \lbrace -2 \rbrace & \text{otherwise} \end{cases} \, \text{,} \] and hence $c \in \lbrace -2, 0 \rbrace$ since the polynomial $M_{4}^{f_{c}}$ splits into linear factors in $R_{D}[\lambda]$ by Corollary~\ref{corollary:multRoots} (see Table~\ref{table:proofSuper}). Thus, the lemma is proved.
\end{proof}

\begin{table}
\caption{Decomposition of $M_{4}^{f_{c}}$ into irreducible factors in $R_{D}[\lambda]$ for the values of $D$ and $c$ appearing in our proof of Lemma~\ref{lemma:proofSuper}.}
\label{table:proofSuper}
$\begin{array}{|c|c|c|}
\hline
D & c & \text{Factorization of } M_{4}^{f_{c}} \text{ in } R_{D}[\lambda]\\
\hline
\hline
1 & \frac{-3}{4} & \lambda^{3} -39 \lambda^{2} +939 \lambda -5221\\
\hline
1 & \frac{1}{2} & \lambda^{3} -44 \lambda^{2} +784 \lambda -8896\\
\hline
2 & \frac{1}{4} & \lambda^{3} -47 \lambda^{2} +779 \lambda -4861\\
\hline
3 & \frac{-1 -3 i \sqrt{3}}{8} & \lambda^{3} +\frac{-109 +3 i \sqrt{3}}{2} \lambda^{2} +\frac{1177 +15 i \sqrt{3}}{2} \lambda -2983 -1218 i \sqrt{3}\\
\hline
3 & \frac{-1 +3 i \sqrt{3}}{8} & \lambda^{3} +\frac{-109 -3 i \sqrt{3}}{2} \lambda^{2} +\frac{1177 -15 i \sqrt{3}}{2} \lambda -2983 +1218 i \sqrt{3}\\
\hline
\end{array}$
\end{table}

By Lemma~\ref{lemma:proofCases} and Lemma~\ref{lemma:proofSuper}, it remains to examine the quadratic rational maps that have a multiple fixed point and whose multipliers lie in the ring of integers of a given imaginary quadratic field. We prove that there is no such map.

\begin{lemma}
\label{lemma:proofSimple}
Assume that $D$ is a positive squarefree integer and $f \colon \widehat{\mathbb{C}} \rightarrow \widehat{\mathbb{C}}$ is a quadratic rational map whose multiplier at each cycle with period less than or equal to $5$ lies in $R_{D}$. Then the fixed points for $f$ are all simple.
\end{lemma}

\begin{proof}
To obtain a contradiction, suppose that $f$ has a multiple fixed point. If $f$ has a unique fixed point, then $f$ is conjugate to $h$ by Proposition~\ref{proposition:quadConj}, and hence \[ M_{5}^{h}(\lambda) = \left( \lambda^{3} -309 \lambda^{2} +27399 \lambda -696691 \right)^{2} \] splits into linear factors in $R_{D}[\lambda]$ by Corollary~\ref{corollary:multRoots}, which is impossible since it is the square of an irreducible polynomial over $\mathbb{Q}$ of degree $3$ and $R_{D}$ is contained in an extension of $\mathbb{Q}$ of degree $2$. Thus, $f$ has exactly two fixed points, and it follows that $f$ is conjugate to $g_{a, 1}$ by Proposition~\ref{proposition:quadConj}, where $a \in R_{D} \setminus \lbrace 1 \rbrace$ is the multiplier of $f$ at its simple fixed point. By Corollary~\ref{corollary:multRoots}, the polynomial \[ M_{3}^{g_{a, 1}}(\lambda) = \lambda^{2} +\left( -4 a^{2} -16 a -18 \right) \lambda +36 a^{3} +112 a^{2} +124 a +89 \] splits into linear factors in $R_{D}[\lambda]$, and hence its discriminant \[ \disc M_{3}^{g_{a, 1}} = 2^{4} (a +2) (a -1)^{3} \] is a square in $R_{D}$. It follows that there exists $b \in R_{D}$ such that $(a -1) (a +2) = b^{2}$, and we have \[ (2 a -2 b +1) (2 a +2 b +1) = 9 \, \text{.} \] Therefore, we have \[ a = \frac{(2 a -2 b +1)^{2} -2 (2 a -2 b +1) +9}{4 (2 a -2 b +1)} \, \text{,} \] which yields \[ a \in \left( R_{D} \setminus \lbrace 1 \rbrace \right) \cap \left\lbrace \frac{c^{2} -2 c +9}{4 c} : c \in R_{D} \text{ and } N(c) \text{ divides } 81 \right\rbrace \, \text{,} \] and hence \[ a \in \begin{cases} \lbrace -3, -2, -1, 0, 2 \rbrace & \text{if } D = 2\\ \left\lbrace -3, -2, \frac{-1 -i \sqrt{3}}{2}, \frac{-1 +i \sqrt{3}}{2}, 2 \right\rbrace & \text{if } D = 3\\ \lbrace -3, -2, 2 \rbrace & \text{otherwise} \end{cases} \, \text{.} \] Note that the polynomial \[ M_{4}^{g_{-3, 1}}(\lambda) = (\lambda -31) \left( \lambda^{2} +80 \lambda +1231 \right) \] does not split into linear factors in $R_{D}[\lambda]$ since it has two non-integer real roots. Moreover, the polynomials \[ M_{4}^{g_{-2, 1}}(\lambda) = \lambda^{3} +9 \lambda^{2} +123 \lambda +1307 \quad \text{and} \quad M_{4}^{g_{2, 1}}(\lambda) = \lambda^{3} -231 \lambda^{2} +17211 \lambda -407861 \] do not split into linear factors in $R_{D}[\lambda]$ either since they are irreducible over $\mathbb{Q}$ of degree $3$ and $R_{D}$ is contained in an extension of $\mathbb{Q}$ of degree $2$. This contradicts the fact that $M_{4}^{g_{a, 1}}$ splits into linear factors in $R_{D}[\lambda]$ by Corollary~\ref{corollary:multRoots} (see Table~\ref{table:proofSimple}). Thus, the lemma is proved.
\end{proof}

\begin{table}
\caption{Decomposition of $M_{4}^{g_{a, 1}}$ into irreducible factors in $R_{D}[\lambda]$ for the values of $D$ and $a$ appearing in our proof of Lemma~\ref{lemma:proofSimple}.}
\label{table:proofSimple}
$\begin{array}{|c|c|c|}
\hline
D & a & \text{Factorization of } M_{4}^{g_{a, 1}} \text{ in } R_{D}[\lambda]\\
\hline
\hline
2 & -1 & \lambda^{3} -15 \lambda^{2} +255 \lambda -1457\\
\hline
2 & 0 & \lambda^{3} -47 \lambda^{2} +779 \lambda -4861\\
\hline
3 & \frac{-1 -i \sqrt{3}}{2} & \lambda^{3} +\left( -21 +14 i \sqrt{3} \right) \lambda^{2} +\left( 99 -124 i \sqrt{3} \right) \lambda -1279 +542 i \sqrt{3}\\
\hline
3 & \frac{-1 +i \sqrt{3}}{2} & \lambda^{3} +\left( -21 -14 i \sqrt{3} \right) \lambda^{2} +\left( 99 +124 i \sqrt{3} \right) \lambda -1279 -542 i \sqrt{3}\\
\hline
\end{array}$
\end{table}

Finally, we have proved Theorem~\ref{theorem:main}, which follows immediately from Lemma~\ref{lemma:proofCases}, Lemma~\ref{lemma:proofSuper} and Lemma~\ref{lemma:proofSimple}.

\providecommand{\bysame}{\leavevmode\hbox to3em{\hrulefill}\thinspace}
\providecommand{\MR}{\relax\ifhmode\unskip\space\fi MR }
\providecommand{\MRhref}[2]{%
	\href{http://www.ams.org/mathscinet-getitem?mr=#1}{#2}
}
\providecommand{\href}[2]{#2}


\begin{thebibliography}{Hug21b}

\bibitem[EvS11]{EvS2011}
Alexandre Eremenko and Sebastian van Strien, \emph{Rational maps with real
multipliers}, Trans. Amer. Math. Soc. \textbf{363} (2011), no.~12,
6453--6463. \MR{2833563}

\bibitem[Hug21a]{H2021a}
Valentin Huguin, \emph{\'{E}tude alg\'{e}brique des points p\'{e}riodiques et
des multiplicateurs d'une fraction rationnelle}, 2021, Thesis
(Ph.D.)--Universit\'{e} Toulouse III - Paul Sabatier.

\bibitem[Hug21b]{H2021b}
\bysame, \emph{Unicritical polynomial maps with rational multipliers}, Conform.
Geom. Dyn. \textbf{25} (2021), 79--87. \MR{4280290}

\bibitem[Mil93]{M1993}
John Milnor, \emph{Geometry and dynamics of quadratic rational maps},
Experiment. Math. \textbf{2} (1993), no.~1, 37--83, With an appendix by the
author and Lei Tan. \MR{1246482}

\bibitem[Mil06]{M2006}
\bysame, \emph{On {L}att\`{e}s maps}, Dynamics on the {R}iemann sphere, Eur.
Math. Soc., Z\"{u}rich, 2006, pp.~9--43. \MR{2348953}

\bibitem[MS95]{MS1995}
Patrick Morton and Joseph~H. Silverman, \emph{Periodic points, multiplicities,
and dynamical units}, J. Reine Angew. Math. \textbf{461} (1995), 81--122.
\MR{1324210}

\bibitem[Sil98]{S1998}
Joseph~H. Silverman, \emph{The space of rational maps on
{$\boldsymbol{P}^{1}$}}, Duke Math. J. \textbf{94} (1998), no.~1, 41--77.
\MR{1635900}

\bibitem[Sil07]{S2007}
\bysame, \emph{The arithmetic of dynamical systems}, Graduate Texts in
Mathematics, vol. 241, Springer, New York, 2007. \MR{2316407}

\end{thebibliography}
\end{document}